\documentclass{amsart}

\usepackage{amsmath, amsthm, amssymb}
\usepackage{epsfig}
\usepackage{graphicx}
\usepackage{array}
\usepackage{amsmath}
\usepackage{amsfonts}
\usepackage{amssymb}%
\usepackage{graphicx}
\usepackage{amssymb}
\usepackage{amsmath,amsxtra,amssymb}

\numberwithin{equation}{section}

\newtheorem{theorem}{Theorem}
\newtheorem{lemma}[theorem]{Lemma}
\newtheorem{remark}[theorem]{Remark}
\newtheorem{corollary}[theorem]{Corollary}
\newtheorem{proposition}[theorem]{Proposition}

\newtheorem{example}[theorem]{Example}

\newcommand{\field}[1]{\mathbb{#1}}
\newcommand{\C}{{\field{C}}}

\newcommand{\ra}{\rightarrow}
\newcommand{\id}{{\iota}}
\newcommand{\ot}{{\otimes}}
\newcommand{\om}{{\omega}}
\newcommand{\tp}{{\widehat{\otimes}}}
\newcommand{\vtp}{{\overline{\otimes}}}

\newcommand{\h}{{\mathcal H}}

\newcommand{\M}{{\mathcal{M}}}
\newcommand{\m}{{\mathfrak{m}}}
\newcommand{\n}{{\mathfrak{n}}}

\newcommand{\B}{{\mathcal{B}}}

\newcommand{\fee}{{\varphi}}

\newcommand{\LL}{{\mathcal{L}^{\infty}(\G)}}
\newcommand{\LO}{{\mathcal{L}^{1}(\G)}}
\newcommand{\LT}{{\mathcal{L}^{2}(\G)}}
\newcommand{\MG}{{\mathcal{M}(\G)}}

\newcommand{\CZ}{{\mathcal{C}_0(\G)}}

\newcommand{\LLL}{{\mathcal{L}^{\infty}(\hat\G)}}

\newcommand{\G}{\mathbb G}
\newcommand{\mcb}{\mathcal{C}^l_{cb}(\LO)}

\title [Left centralizers and actions of l.c.q. groups]
{Representation of left centralizers for actions of locally compact quantum groups}
\author{Mehrdad Kalantar}
\address{ School of Mathematics and Statistics,
Carleton University, Ottawa, ON, Canada}
\email{mkalanta@math.carleton.ca}


\begin{document}

\begin{abstract}
We generalize the representation theorem of Junge, Neufang and Ruan in \cite{JNR},
and some of the important results which were used in its proof,
to the case of actions of locally compact quantum groups on von Neumann algebras.
\end{abstract}
\maketitle

\section{Introduction and Preliminaries}
In \cite{JNR} the authors
have investigated the quantum group analogue of the class of completely bounded
multiplier algebras which play an important role in Fourier analysis over groups, by means of a
representation theorem.
They, indeed, generalized and unified earlier results in the commutative \cite{N}
and co-commutative \cite{NRS} settings.
Beside its elegance, their representation theorem has proven to have very important applications.
In particular, this result enables them
to express quantum group duality precisely in terms of a commutation relation.
Moreover, it provides a systematic way of producing
interesting and non-trivial quantum channels, and more surprisingly, to explicitly
calculate some important quantities associated to them,
such as the c.b. minimal entropy \cite{JNR-QIT}.

All this has motivated us to investigate this representation theorem further.
We see that the main point behind these results is the fact that a locally compact (quantum)
group acts on itself, by left (right) multiplication.
Therefore, we consider the comultiplication as a left action of a locally compact quantum group $\G$
on itself, then replace $\LL$ by a von Neumann algebra on which $\G$ acts, and discuss the
representation results.

First, let us recall some definitions and preliminary results that we will be using in this paper.
For more details on locally compact quantum groups we refer the reader to \cite{KV}.

A {\it locally compact quantum group} $\G$ is a
quadruple $(\LL, \Gamma, \varphi, \psi)$, where $\LL$ is a
von Neumann algebra,
$\Gamma: \LL\to \LL \vtp \LL$
is a co-associative co-multiplication,
and $\varphi$ and  $\psi$ are  (normal faithful semi-finite) left and right
Haar weights on $\LL$, respectively.

The \emph{reduced quantum group $C^*$-algebra} associated to $\G$ is denoted by $\CZ$, which
is a weak$^*$ dense $C^*$-subalgebra of $\LL$.
Let $\MG$ denote the dual space $\CZ^{*}$.
There exists  a completely contractive multiplication on $\MG$ given by
the convolution
\[
\star : \MG\tp \MG\ni \mu\ot \nu \,\longmapsto\, \mu \star \nu 
= \mu (\id\otimes \nu)\Gamma = \nu (\mu \otimes \id)\Gamma
\in \MG
\]
such that  $\MG$ contains $\LO:= \LL_*$ as a norm closed two-sided ideal.
Therefore, for each $\om\in\MG$, we obtain a completely bounded map
\begin{equation}
\label {F.multi}
\mathfrak{m}_{\omega}(f) = \om \star f
\end{equation}
on $\LO$ with $\|\mathfrak{m}_{\om}\|_{cb}\le \|\omega\|$.

A linear map $\m$ on $\LO$ is called a {\it left centralizer}
of $\LO$ if it satisfies
\[
\m(f\star g) \,=\, \m(f)\star g
\]
for all $f, g \in \LO$.
We denote by $\mcb$ the space of all
completely bounded left centralizers of $\LO$.

It turns out that the map $\mathfrak{m}_{\omega}$
in (\ref{F.multi}) is a left centralizer of $\LO$.
Therefore, we obtain the natural inclusion
\begin{equation}
\label{F.inclusion}
\LO \hookrightarrow \MG \hookrightarrow \mcb,
\end{equation}
where the first inclusion is completely isometric homomorphisms, and
the second one is a completely contractive homomorphism.
These algebras are typically not equal.
We have
\begin{equation}
\label {F.equi}
\MG =  \mcb
\end{equation}
if and only if $\G$ is co-amenable, i.e., $\LO$  has a contractive (or bounded)
approximate identity.

Let $N$ be a von Neumann algebra, a left action of $\G$ on $N$ is a $*$-isomorphism
$\alpha: N \ra \LL\vtp N$ such that
\[
(\Gamma\otimes\id)\,\alpha\,=\,(\id\otimes\alpha)\,\alpha\,.
\]
The action $\alpha$ is called ergodic if the fixed point algebra
$$N^\alpha\,:=\,\{\,x\,\in\,N\,:\,\alpha(x)\,=\,1\otimes x\,\}$$
is trivial.
We say that $\alpha$ is {\it faithful} if
$$\big\{\,(\iota\otimes \omega)\,\alpha(x)\, :\,  \omega\,\in \,N_*, \, x\,\in\, N\,\big\}$$
spans a weak* dense subspace of $\LL$, and
we say that $\alpha$ is {\it strongly faithful} if
$$\big\{\,(\iota\otimes \omega)\,\alpha(x)\, :\,  \omega\,\in \,N_*, \, x\,\in\, N,\,\|\om\|=\|x\|=1\,\big\}$$
is weak* dense in the unit ball of $\LL$.
$\alpha$ is said to be {\it co-faithful} if the span of
$$\big\{\,(f\otimes\iota)\,\alpha(x)\, :\,  f\,\in\, \LO,\, x\,\in\, N\,\big\}$$
is weak* dense in $N$.

\section{The Representation Theorem}
Consider an action $\alpha$ of a locally compact quantum group $\G$
on a von Neumann algebra $N$.
Let $\theta$ be an normal semi-finite faithful weight on $N$, and let $K$ be the $GNS$ Hilbert space of $\theta$.
It is proved in \cite[Theorem 4.4]{V} that $\alpha$
is implemented by a unitary $U_\alpha\in \LL\vtp B(K)$, that is
\begin{equation}\label{1212}
\alpha(x)\, = \,U_\alpha\,(1\otimes x)\,U_\alpha^* \ \ \ (x\in N)\,,
\end{equation}
such that this unitary is a corepresentation of $\G$, i.e.,
\begin{equation}\label{4}
(\Gamma\otimes\iota)\, U_\alpha\, = \,(U_\alpha)_{23}\,(U_\alpha)_{13}\,.
\end{equation}
By \cite[Proposition 3.7]{V}, the unitary $U_\alpha$ also satisfies the following
\begin{equation}\label{00001}
U_\alpha\,(\hat J\otimes J_\theta)\, = \,(\hat J\otimes J_\theta)\,U_\alpha^*\,,
\end{equation}
where $\hat J$ and $J_\theta$ are the modular conjugations of the dual Haar weight $\hat \fee$, and $\theta$,
respectively.
Using (\ref{1212}), we can extend $\alpha$
to an action
$$\tilde\alpha : B(K)\,\rightarrow\,\LL\,\vtp\, B(K)$$
of $\G$ on $B(K)$.
Then it is obvious that $U_\alpha\in \LL\vtp (B(K)^{\tilde\alpha})'$.
Now, for $f\in\LO$, let $\lambda_\alpha(f) := (f\otimes\iota)U_\alpha\in B(K)$, and define
\[
\hat{N}\, :=\, \{\,\lambda_\alpha(f)\, : \,f\in \LO\,\}''\,.
\]
Then one can easily conclude the following.
\begin{proposition}\label{3}
We have $\hat N' = B(K)^{\tilde\alpha}$, and therefore $U_\alpha\in \LL\vtp\hat N$.
\end{proposition}

If $\m\in\mcb$ is a left centralizer, then $\m^*$ defines a c.b. normal map on $\LL$.
One of the main points in the representation theorem of \cite{JNR} was to find a canonical
extension of $\m^*$ to $\mathcal{CB}^\sigma(B(\LT))$.
But the first obstacle that arises in the attempt of generalizing the representation theorem
to the case of $\G$-spaces is that in this setting there is not, in general,
any canonical way to assign to $\m\in\mcb$ a c.b. normal map on $N$;
if $\G$ is commutative, i.e., a locally compact group $G$,
then for $\mu\in \M(G) = \mathcal{C}_{cb}^l(\mathcal{L}^1(G))$,
and an action $\alpha:G\to Aut(N)$
we can define the map $\Theta_\alpha(\mu)$ on $N$ by
\begin{equation}\label{000}
\Theta_\alpha(\mu)\,(x)\,=\, \int_G\, \alpha_{s^{-1}}(x)\,d\mu(s)\ \ \
(x\in N)\,.
\end{equation}
Also, in the case of actions of locally compact quantum groups $\G$,
for $f\in\LO$ we can define $\Theta_\alpha(f)$ on $N$ by
\begin{equation}\label{0000}
\Theta_\alpha(f)\,=\, (f\otimes\id)\,\alpha\,.
\end{equation}
The formula (\ref{0000}) coincides with (\ref{000}) when $\G=G$,
and both coincide with $\m^*$, when $N = \LL$.
So, before anything, we have to find $\Theta_\alpha(\m)$ on $N$, analogues to the above
maps, for the case of a quantum action $\alpha$ and arbitrary $\m\in\mcb$.

\begin{theorem}\label{5}
Let $\m\in\mcb$, then there exists a unique completely bounded normal map $\Theta_\alpha(\m)$ on $N$ such that
\begin{equation}\label{a1}
\alpha\,\Theta_\alpha(\m)\, =\, (\m^*\otimes\iota)\,\alpha\,.
\end{equation}
Moreover, if $f\in\LO$ then we have $\Theta_\alpha(\m_f)= (f\otimes\id)\,\alpha$.
\end{theorem}
\begin{proof}
Let $x\in N$.
Then we have
\begin{eqnarray*}
(\Gamma\otimes\iota)\,((\m^*\otimes\iota)\,\alpha(x))
&=& (\m^*\otimes\iota\otimes\iota)\,((\Gamma\otimes\iota)\,\alpha(x))\\
&=& (\m^*\otimes\iota\otimes\iota)\,((\iota\otimes\alpha)\,\alpha(x))\\
&=& (\iota\otimes\alpha)\,(\m^*\otimes\iota)\,\alpha(x))\,.
\end{eqnarray*}
Hence, by \cite[Theorem 2.7]{V}, we get
$$(\m^*\otimes\iota)\,\alpha(x)\,\in\, \alpha(N)\,.$$
Therefore we can define a map $\Theta_\alpha(\m) : N \ra N$ by
\begin{equation}\label{1213}
\Theta_\alpha(\m)\, =\, \alpha^{-1}\,(\m^*\otimes\iota)\,\alpha\,.
\end{equation}
Since $\alpha$ is an $*$-isomorphism and $\m^*$ is a completely bounded normal map,
we conclude that $\Theta_\alpha(\m)$ is a completely bounded normal map on $N$.
Uniqueness follows from (\ref{a1}), which then implies the last part of the theorem
since for $f\in\LO$ we have $\alpha((f\otimes\id)\alpha) = (\m_f^*\otimes\iota)\,\alpha\,.$
\end{proof}

\begin{proposition}
Suppose that $\alpha$ is faithful. If
$$\alpha\,\Phi\, =\, (\Psi\otimes\iota)\,\alpha$$
for some completely bounded normal maps $\Phi: N\ra N$ and $\Psi: \LL\ra \LL$,
then there exists $\m\in \mcb$ such that $\Psi=\m^*$
and $\Phi=\Theta_\alpha(\m)$.
\end{proposition}
\begin{proof}
We have
\begin{eqnarray*}
(\Gamma\,\Psi\otimes\iota)\,\alpha &=& (\Gamma\otimes\iota)\,\alpha\,\Phi
\,=\, (\iota\otimes\alpha)\,\alpha\,\Phi
\,=\, (\iota\otimes\alpha)\,((\Psi\otimes\iota)\,\alpha) \\
&=& (\Psi\otimes\iota\otimes\iota)\,((\iota\otimes\alpha)\,\alpha)
\,=\, (\Psi\otimes\iota\otimes\iota)\,(\Gamma\otimes\iota)\,\alpha\\
&=& ((\Psi\otimes\iota)\Gamma\,\otimes\,\iota)\,\alpha\,,
\end{eqnarray*}
which implies, since $\alpha$ is faithful, that
$$\Gamma\,\Psi \,=\, (\Psi\otimes\iota)\Gamma.$$
So, there exists $\m\in\mcb$ such that $\Psi = \m^*$.
Moreover, we have
\[
\alpha\,\Theta_\alpha(\m)=(\m^*\otimes\iota)\,\alpha \,=\,
(\Psi\otimes\iota)\,\alpha \,=\, \alpha\,\Phi\,.
\]
Since $\alpha$ is injective, it follows that $\Phi = \Theta_\alpha(\m)$.
\end{proof}

The following theorem is our main result, which generalizes
the main representation theorem of \cite{JNR}.
Following their notation, if $X , Y\subseteq \B(H)$, we denote by $\mathcal {CB}^{\sigma, X}_Y(\B(H))$
the algebra of all normal completely bounded $Y$-bimodule maps $\Phi$ on
$\B(H)$ that leave $X$ invariant.


\begin{theorem}\label{00010}
Let $\alpha$ be a left action of a locally compact quantum group $\G$
on a von Neumann algebra $N$. Then there exists a completely contractive anti-homomorphism
$$\tilde\Theta_\alpha : \mcb \rightarrow \mathcal{CB}^{\sigma,N}_{\hat N'}(B(K))$$
such that $\tilde\Theta_\alpha(\m)$ extends $\Theta_\alpha(\m)$, for every $\m\in\mcb$.
Moreover, If $\alpha$ is strongly faithful, the map $\tilde\Theta_\alpha$
is completely isometric.
\end{theorem}
\begin{proof}
By Theorem \ref{5} (considered for the action $\tilde{\alpha}$),
for every $\m\in\mcb$ there exists a normal map
$\tilde\Theta_{\alpha}(\m) := \Theta_{\tilde\alpha}(\m)$ on $B(K)$ such that
\[
\tilde\alpha\,\tilde\Theta_{\alpha}(\m)\, =\, (\m^*\otimes\iota)\,\tilde\alpha\,.
\]
Then, for $y\in N$ we obtain
\begin{eqnarray*}
\tilde\alpha\,\tilde\Theta_{\alpha}(\m)\,(y) \,=\, (\m^*\otimes\iota)\,\tilde\alpha(y)
\,=\, (\m^*\otimes\iota)\,\alpha(y)
\,=\, \alpha\,\Theta_{\alpha}(\m)\,(y)\,,
\end{eqnarray*}
which shows that $\tilde\Theta_{\alpha}(\m)(N)\subseteq N$ and
\[
\tilde\Theta_{\alpha}(\m)\,(y)\, =\, \Theta_{\alpha}(\m)\,(y)
\]
for all $y\in N$.
Next, let $z\in B(K)$ and $\hat y\in\hat N' = B(K)^{\tilde\alpha}$, then we have
\begin{eqnarray*}
\tilde\alpha(\tilde\Theta_{\alpha}(\m)(z \hat y)) &=& (\m^*\otimes\iota)\,\tilde\alpha(z\hat y)
\,=\, (\m^*\otimes\iota)\,(\tilde\alpha(z)(1\otimes \hat y))\\
&=& ((\m^*\otimes\iota)\tilde\alpha(z))\,(1\otimes \hat y)
\,=\, \tilde\alpha\,(\tilde\Theta_{\alpha}(\m)(z)\,\hat y)\,.
\end{eqnarray*}
Since $\tilde\alpha$ is injective, this implies
$
\tilde\Theta_{\alpha}(\m)(z \hat y) = \tilde\Theta_{\alpha}(\m)(z)\, \hat y.
$
Similarly we obtain
\[
\tilde\Theta_{\alpha}(\m)(\hat y z)\, =\, \hat y\,\tilde\Theta_{\alpha}(\m)(z)\,.
\]
Hence, for all $\m\in\mcb$ we obtain
\[
\tilde\Theta_\alpha(\m)\,\in\, \mathcal{CB}^{\sigma,N}_{\hat N'}(B(K))\,.
\]
It is obvious from (\ref{1213}) that $\tilde\Theta_\alpha$ is linear. Moreover,
let $\m, \n \in \mcb$, then we have
\begin{eqnarray*}
\alpha\,(\tilde\Theta_\alpha(\m\star \n)) &=& ((\m\star\n)^*\,\otimes\iota)\,\alpha 
\,=\, (\n^*\otimes\iota)\,(\m^*\otimes\iota)\,\alpha \\
&=& (\n^*\otimes\iota)\,\alpha(\tilde\Theta_\alpha(\m))
\,=\, \alpha\,(\tilde\Theta_\alpha(\n)\,\tilde\Theta_\alpha(\m))\,.
\end{eqnarray*}
So $\tilde\Theta_\alpha(\m\star\n) = \tilde\Theta_\alpha(\n)\,\tilde\Theta_\alpha(\m)$, by injectivity of $\alpha$.

It is seen from (\ref{1213}), that $\Theta_\alpha$ and $\tilde\Theta_\alpha$ are completely contractive maps.
Now, suppose that $\alpha$ is strongly faithful, then for $\m\in\mcb$ we obtain
\begin{eqnarray*}
\|\,\Theta_\alpha(\m)\,\|
&=&
\sup\,\{\,\|\Theta_\alpha(\m)(x)\|\,:\,x\in (N)_1\,\}\\
&=&
\sup\,\{\,\|\alpha\,\Theta_\alpha(\m)(x)\|\,:\,x\in (N)_1\,\}\\
&=&
\sup\,\{\,\|(\m^*\otimes\id)\,\alpha(x)\|\,:\,x\in (N)_1\,\}\\
&\geq&
\sup\,\{\,|\,(f\otimes\om)\,(\m^*\otimes\id)\,\alpha(x)\,|\,:\,f\in\LO_1,\om\in (N_*)_1,x\in (N)_1\,\}\\
&=&
\sup\,\{\,|\,(\m(f)\otimes\om)\,\alpha(x)\,|\,:\,f\in\LO_1,\om\in (N_*)_1,x\in (N)_1\,\}\\
&=&
\sup\,\{\,\|(\m(f)\|\,:\,f\in\LO_1\,\}\\
&=&
\|\,\m\,\|\,,
\end{eqnarray*}
which shows that $\Theta_\alpha$ is an isometry.
In fact, similarly we can see that $\Theta_\alpha$ is completely isometric.
Then, since $\tilde\Theta_\alpha(\m)$
extends $\Theta_\alpha(\m)$ for every $\m\in\mcb$, we conclude that $\tilde\Theta_\alpha$
is completely isometric.
\end{proof}

\begin{remark}
\emph{The representation map in \cite{JNR} was also proved to be onto. But this is not the case in
the setting of left actions, as the following example shows.
It seems that the surjectivity, results from the fact that $\G$
acts on itself from both left and right, and that left centralizers
are the commutant of the right centralizers, and vice versa.
This is also suggested by the following example.}
\end{remark}
\noindent
\begin{example}
\emph{Consider the action of the additive group of integers $(\mathbb Z\,,\,+)$
on the unit circle $\mathbb T$ by irrational rotations. This
action is known to be ergodic.
For $h\in L^\infty(\mathbb{T})$ and $f\in L^1(\mathbb{T})$ we have
\begin{equation}\label{a2}
(\id \otimes f)\,\alpha(h) (n)\,=\, \int_{\mathbb{T}}\,f(s)\ \alpha_{n}^{-1}(h)(s)\ d s\,.
\end{equation}
Now consider, for example, the Dirac function
$\delta_0\in l^\infty(\mathbb Z)$, and suppose that $\mathcal{F}\subset\mathbb{Z}$
is a finite subset not including $0$. Then choose a closed interval $I$ around $1\in\mathbb{T}$
such that $I \cap \alpha_n^{-1}I = \emptyset$ for all $n\in \mathcal{F}$.
Denote by $\chi_I$ the characteristic function of $I$,
then for $\displaystyle f=\frac{1}{length(I)}\,\chi_I$ and $h=\chi_I$ in
(\ref{a2}) we get $(\id \otimes f)\,\alpha(h) (0) = 1$,
and $(\id \otimes f)\,\alpha(h) (n) = 0$ for all $n\in\mathcal{F}$.
This shows that we can choose a sequence $(f_k)\subset L^1(\mathbb T)$ and
$(h_k)\subset L^\infty(\mathbb T)$ such that $(\id \otimes f_k)\,\alpha(h_k)\to \delta_0$
in weak* topology of $l^\infty(\mathbb Z)$.
Similarly, we can approximate any finite support function in $l^\infty(\mathbb Z)$,
whence any function in $c_0(\mathbb Z)$, whose norm is bounded by 1,
by a sequence of the form $(\id \otimes f_k)\,\alpha(h_k)$,
with $(f_k)$ and $(h_k)$ sequences in unit balls of $L^1(\mathbb T)$ and $L^\infty(\mathbb T)$, respectively.
Since the unit ball of $c_0(\mathbb Z)$ is weak* dense in the unit ball of $l^\infty(\mathbb Z)$,
it follows that the action $\alpha$ is strongly faithful.}

\emph{Now, for the Dirac measure $\delta_{-1}$
concentrated on $-1\in \mathbb T$, consider the map
$$\tilde\Theta_\alpha(\delta_{-1})\,\in\,
\mathcal{CB}^{\sigma,\mathcal{L}^\infty(\mathbb T)}_{\mathcal{R}(\mathbb T)}(B(\mathcal{L}^2(\mathbb T)))\,.$$
Then, obviously there cannot be $\m\in\mathcal{C}_{cb}^l({l}^1(\mathbb Z)) = l^1(\mathbb Z)$ such that
$\tilde\Theta_\alpha(\delta_{-1}) = \tilde\Theta_\alpha(\m)$, since $\tilde\Theta_\alpha$ is injective \cite{JNR}.
Hence, $\tilde\Theta_\alpha$ is not onto.}
\end{example}

The following is a consequence of Theorem \ref{00010}.
\begin{corollary}\label{0004}
For every $z\in B(K)$ we have
\[
U_\alpha^*\,\big((\m^*\otimes\iota)\,\tilde\alpha(z)\big)\,U_\alpha\,\in 1\otimes B(K)\,.
\]
\end{corollary}

\begin{remark}
\emph{In \cite{JNR}, similar result as Corollary \ref{0004} was proved directly,
then used to define the canonical extension map $\tilde\Theta$.}
\end{remark}

In the rest of the paper, we look at some of the main results that have been used in the proof
of the representation theorem of \cite{JNR}, and will prove analogue results in our setting.
\begin{proposition}
Define $\alpha' : N'\ra \LL\vtp N'$ by
\begin{equation}\label{1112}
\alpha'(x') \,:=\, (\hat J\otimes J_\theta)\,\alpha(J_\theta x' J_\theta)\,(\hat J\otimes J_\theta)\ \ (x'\in N')\,.
\end{equation}
Then $\alpha'$ defines an action of $\G$ on $N'$.
\end{proposition}
\begin{proof}
We easily see that $\alpha'$ is a $*$-isomorphism. Moreover,
using (\ref{00001}) we have
\begin{eqnarray*}
\alpha'(x') &=&
(\hat J\otimes J_\theta)\,U_\alpha\,(1\otimes J_\theta x' J_\theta)\,U_\alpha^*\,(\hat J\otimes J_\theta)\\
&=& (\hat J\otimes J_\theta)\,U_\alpha\,(\hat J\otimes J_\theta)
(1\otimes x')(\hat J\otimes J_\theta)\,U_\alpha^*\,(\hat J\otimes J_\theta)\\
&=& U_\alpha^*\,(1\otimes x')\,U_\alpha\,.
\end{eqnarray*}
Hence, we get
\begin{eqnarray*}
(\iota\otimes\alpha')\,\alpha'(x') &=&
(U_\alpha^*)_{23}\,(U_\alpha^*)_{13}\,(1\otimes 1 \otimes x')\,(U_\alpha)_{13}\,({U_\alpha})_{23}\\
&=& (\Gamma\otimes\iota)\,(U_\alpha^*\,(1\otimes x')\,U_\alpha)
\,=\, (\Gamma\otimes\iota)\,\alpha'(x')\,.
\end{eqnarray*}
\end{proof}
From the above proof, we see that $\alpha'$ can be extended to an action $\tilde\alpha'$
of $\G$ on $B(K)$, defined as
\[
\tilde\alpha'(z)\,:=\,U_\alpha^*\,(1\otimes z)\,U_\alpha\ \ \ (z\in B(K))\,.
\]
(Note that $U_\alpha^*$ is not the canonical unitary implementation of $\alpha'$).
Then we can easily derive the following.
\begin{lemma}\label{6}
We have
\begin{enumerate}
\item
$(N')^{\alpha'} = J_\theta N^\alpha J_\theta$;
\item
$B(K)^{\tilde\alpha'} = J_\theta B(K)^{\tilde\alpha} J_\theta$.
\end{enumerate}
\end{lemma}

If $R$ denotes the unitary antipode of a locally compact quantum group $\G$,
then we have
$
R(x)=\hat J\,x\,\hat J\,,
$
which is a result of the fact that $\hat J\,\LL\,\hat J\,=\,\LL$.
The following theorem gives a generalization of this fact, whence, a way to define
a `unitary antipode' on the `dual' of a $\G$-space.

\begin{theorem}\label{06}
We have
\[
J_\theta\, \hat N\, J_\theta\,=\, \hat N\,.
\]
\end{theorem}
\begin{proof}
Similarly to the proof of Proposition \ref{3} we can show that
\begin{equation}\label{0001}
B(K)^{\tilde\alpha'}\, = \, \{\,(f\otimes \id)U_\alpha^*\,:\,f\in\LO\,\}^{'}\,.
\end{equation}
Now, the left hand side of (\ref{0001}) is equal to $J_\theta B(K)^{\tilde\alpha} J_\theta$
by Lemma \ref{6},
and the right hand side is equal to $\hat N'$ by the definition.
Hence, from Proposition \ref{3} we conclude
\[
J_\theta\, \hat N'\, J_\theta\, = \, J_\theta\, B(K)^{\tilde\alpha}\, J_\theta\,=\, \hat N'\,.
\]
Now, let $\hat x\in\hat N$, then for every $\hat y'\in\hat N'$ we obtain
\[
J_\theta\, \hat x\, J_\theta\,\hat y'\,=\,J_\theta\, \hat x\, J_\theta\,\hat y'\,J_\theta\,J_\theta\,=\,
J_\theta\, J_\theta\,\hat y'\,J_\theta\, \hat x\,J_\theta\,=\,
\hat y'\,J_\theta\, \hat x\,J_\theta\,,
\]
which implies $J_\theta\, \hat x\,J_\theta\,\in\hat N$.
\end{proof}

It is known for a locally compact quantum group $\G$ that
\begin{eqnarray*}
\langle\, \LL\,\LLL\,\rangle &=& \langle\, \LL\,\LLL'\,\rangle
\,=\,\langle\, \LL'\,\LLL\,\rangle \\ &=&\langle\, \LL'\,\LLL'\,\rangle \,=\, B(\LT)\,,
\end{eqnarray*}
which becomes a very powerful tool when dealing with normal maps on $B(\LT)$.
This result is also crucial in the proof of main results in \cite{JNR}.
In the following, we prove a generalization of this result.
\begin{theorem}\label{2}
If $\alpha$ is ergodic and co-faithful, then we have
\[
\langle \,N\,\hat N\,\rangle\, =\, B(K)\,,
\]
where $\langle N\hat N\rangle$ is the weak* closure of the linear span of the set
$\{\,y\,\hat y \,: \,y\in N, \hat y\in \hat N\,\}$.
\end{theorem}
\begin{proof} Let $f,g\in \LO$ and $y\in N$, using (\ref{4}) we have
\begin{eqnarray*}
\big((f\otimes\iota) \,U_\alpha\big)\,\big((g\otimes\iota)\,\alpha(y)\big) &=&
\big((f\otimes\iota)\,U_\alpha\big)\,\big((g\otimes\iota)(U_\alpha(1\otimes y)U_\alpha^*)\big)\\
&=& (g\otimes f\otimes\iota)\,\big((U_\alpha)_{23}\,(U_\alpha)_{13}\,(1\otimes 1\otimes y)\,(U_\alpha^*)_{13}\big)\\
&=& (g\otimes f\otimes\iota)\,\big(((\Gamma\otimes\iota)U_\alpha)\,(1\otimes 1\otimes y)\,
((\Gamma\otimes\iota)U_\alpha^*)(U_\alpha)_{23}\big)\\
&=& (g\otimes f\otimes\iota)\,\big(((\Gamma\otimes\iota)(U_\alpha(1\otimes y)U_\alpha^*))\,(U_\alpha)_{23}\big)\\
&=& (g\otimes f\otimes\iota)\,\big(((\Gamma\otimes\iota)\alpha(y))\,(U_\alpha)_{23}\big).
\end{eqnarray*}
We claim that the last term is in $N\hat N$. To see this, let $\hat y\in\hat N$ and $x\in \LL$, then
\begin{eqnarray*}
(g\otimes f\otimes\iota)\,\big(((\Gamma\otimes\id)\alpha(y))\,(1 \otimes x\otimes \hat y)\big)
&=& (g \star (x\,f)\otimes\id)\,(\alpha(y)\,(1\otimes\hat y))\\
&=& \big((g \star (x\,f)\otimes\iota)\,\alpha(y)\big)\,\hat y\, \in\, N\,\hat N.
\end{eqnarray*}
Since the span of
$$\{\,1\otimes x\otimes \hat y \,:\, \hat y\in\hat N,\, x\in \LL\,\}$$
is weakly dense in $\C\otimes \LL\vtp \hat N$, and $(U_\alpha)_{23}\in \C\otimes \LL\vtp \hat N$,
the claim follows.
Then, since $\alpha$ is co-faithful, it yields that
$\hat N N\subseteq N\hat N$, which implies that $\langle N\hat N\rangle$ is a von Neumann algebra.
Moreover, its commutant is $N'\cap(\hat N)'$. By proposition \ref{3} we have
$(\hat N)' = B(K)^{\tilde\alpha'}$, and since $\alpha$ is ergodic, by Lemma \ref{6} we have
\[
N'\cap(\hat N)' \,=\, N' \cap B(K)^{\tilde\alpha'}\, =\, (N')^{\alpha'}\, =\, J_\theta N^\alpha J_\theta \,=\, \C 1\,,
\]
which completes the proof.
\end{proof}

Combining Theorems \ref{06} and \ref{2}, we obtain the following.
\begin{corollary}
If $\alpha$ is ergodic and co-faithful, then we have
\[
\langle\, N'\,\hat N\,\rangle\, =\, B(K)\,.
\]
\end{corollary}

We note that the co-faithfulness condition assumed in Theorem \ref{2}
is a very weak condition. In fact, any action of any discrete group on a von Neumann algebra $N$
is co-faithful. This is clear since $(\delta_e\otimes\id)\alpha$ is the identity map
on $N$. In particular, for the action of integers on the unit circle
discussed above, all the ingredients for the proof of the representation theorem of
\cite{JNR} hold, which yet, prove to be insufficient for surjectivity.
Moreover, using Theorem \ref{00010} one can easily see that
any action of a discrete quantum group is co-faithful.

Also, boundary actions provide
(non-trivial) examples of quantum group actions $\alpha$ who satisfy conditions of Theorem \ref{2}.
Let $\mu\in\MG$ be a state, the (noncommutative) Poisson boundary of $\mu$ is defined as
\[
 \h_\mu\,:=\,\{\,x\in\LL\,:\, \mathfrak{m}_\mu^*(x)=x\,\},
\]
where $\mathfrak{m}_\mu$ is the convolution operator (see \ref{F.multi}).
Then $\h_\mu$ is a weak* closed operator system in $\LL$, and
moreover, it is a von Neumann algebra with the Choi--Effros product
(for more details we refer the reader to \cite{I} and \cite{KNR}).
Then the restriction of $\Gamma$ defines an action of $\G$ on $\h_\mu$.
This action is clearly ergodic since it is the restriction of $\Gamma$.
Moreover, if $\G$ is co-amenable, and $(e_i)$ is a bounded approximate identity
for $\LO$, then $(e_i\ot\id) \Gamma(x) \rightarrow x$ in the weak* topology for all $x\in\LL$,
and so it follows that in this case the action is also co-faithful.
%

\begin{remark}
\emph{As an application of Theorem \ref{00010}, in the finite dimensional case,
we can produce a class of quantum channels coming from actions of
quantum groups on finite dimensional von Neumann algebras.
It will be interesting to investigate the properties of these quantum channels,
similarly to \cite{JNR-QIT}.}
\end{remark}

\bibliographystyle{plain}

\end{document}